\documentclass[12pt]{amsart}
\usepackage{amssymb,latexsym}

\theoremstyle{definition}
\newtheorem{theorem}{Theorem}
\newtheorem{corollary}[theorem]{Corollary}
\newtheorem{proposition}[theorem]{Proposition}
\newtheorem{lemma}[theorem]{Lemma}
\newtheorem{definition}[theorem]{Definition}
\newtheorem{example}[theorem]{Example}
\newtheorem{notation}[theorem]{Notation}

\newtheorem{remark}[theorem]{Remark}

\newtheorem*{theoremmm}{Theorem}

\newcommand{\numberset}{\mathbb}

\newcommand{\Z}{\numberset{Z}}

\newcommand{\F}{\numberset{F}}

\newcommand{\Pro}{\numberset{P}}

\newcommand{\Ol}{\mathcal{O}}

\usepackage{hyperref}

\usepackage[margin=2.4cm]{geometry}
\usepackage[all,cmtip]{xy}

\usepackage{fouriernc}

\usepackage{amsaddr}

\begin{document}

\title[]{Embedding Suzuki curves in $\Pro^4$}

\author{Edoardo Ballico$^1$}
\address{Department of Mathematics, University of Trento\\ Via Sommarive 14,
38123 Povo (Trento), Italy}

\author{Alberto Ravagnani$^2$}
\address{Institut de Math\'{e}matiques, Universit\'{e} de Neuch\^{a}tel\\Rue
Emile-Argand 11, CH-2000 Neuch\^{a}tel, Switzerland}
\email{$^1$edoardo.ballico@unitn.it}
\email{$^2$alberto.ravagnani@unine.ch}


\begin{abstract}
Here we study the projective geometry of smooth models $X_n \subseteq \Pro^4$ of
plane
Suzuki curves $S_n$. The knowledge of a system of generators for the
Weierstrass semigroup
 at the only singular point of the curve is shown to have
relevant geometric
consequences. In particular, here we explicitly  count the hypersurfaces of
$\Pro^4$ containing $X_n$
and provide a geometric characterization of those of small degree. We prove that
the
characterization cannot 
be extended to higher-degree hypersurfaces of $\Pro^4$. 
\end{abstract}


\maketitle

\section{Introduction}\label{intr0}
Let $n \ge 2$ be an integer and let $q_0$ and $q$ be defined by $q_0:=2^n$,
$q:=2q_0^2$.
Let  $\F_q$ denote the finite field with $q$ elements and fix any field $\F$
containing $\F_q$. For the rest of the
paper, $\F$ will be the base field.
Given an integer $r>0$, we denote by $\Pro^r$ the $r$-dimensional
projective space over $\F$. The projective plane $\Pro^2$ will be referred
to homogeneous coordinates $(x:y:z)$.

The \textbf{Suzuki curve} $S_n \subseteq
\Pro^2$ associated to the integer
$n$ is defined over $\F$ by the following affine equation: 
$$y^q-y=x^{q_0}(x^q-x)$$
(see \cite{HKT}, Example 5.24). This curve is known to have only one point 
lying on the
hyperplane at infinity $\{ z=0\}$, namely, $P_\infty:=(0:1:0)$.
This point, at which $S_n$ has a cusp, is also the only singular point of
the curve.
The genus of $S_n$ (i.e., by definition, the geometric genus of its
normalization)
is known to be $g_n:=q_0(q-1)$.

\subsection{Main references on Suzuki curves}

Suzuki curves are studied in depth
throughout the book \cite{HKT}. The are very interesting from a geometric viewpoint 
because of their optimality
(Chapter 10) and their large group of automorphisms (Theorem 11.127 and, more
generally, Section 12.2). Relevant properties of the Suzuki group date back to
\cite{He}. A comprehensive view on Suzuki curves and their quotients is given in
\cite{Giu}. On the same topics see also \cite{wi1} and \cite{wi2}, Chapter V. 
More recently, the $p$-torsion group scheme of Jacobians of Suzuki curves has
been studied in \cite{AMS}.

Interesting applications of Suzuki curves in geometric Coding
Theory have
been successfully considered in \cite{HH} and in \cite{MA}, computing also the
Weierstrass
semigroup associated to pairs of points of $S_n$ (\cite{MA}, Section III).

\subsection{Layout of the paper}

Here we consider a Suzuki curve $S_n$, as defined above, and its normalization
$\pi:C_n \to S_n$. The normalization morphism, $\pi$, is known to be injective.
In Section \ref{sec1} we study linear systems of the form $|m\pi^{-1}(P_\infty)|$, $m
\in \Z_{\ge 0}$. In particular, we give necessary and sufficient conditions for
$|m\pi^{-1}(P_\infty)|$ to be very ample. The smallest integer $m$ with this property is
$q+2q_0+1$. Moreover, the morphism induced by $|(q+2q_0+1)\pi^{-1}(P_\infty)|$ embeds
$C_n$ into $\Pro^4$. The curve obtained in this way, denoted by $X_n$, is a smooth model
of $S_n$ in $\Pro^4$. The
goal of the paper is to study the projective geometry of $X_n$. More
precisely, we are interested in explicitely counting the 
hypersurfaces of $\Pro^4$ containing $X_n$ and describing those of small
degree. Our main result is the following
one.

\begin{theoremmm}[see Theorem \ref{ttt} and Corollary \ref{bombetta}]
Let $X_n$ be the curve defined above and let $g_n=q_0(q-1)$ be its genus. The following 
facts hold.
\begin{enumerate}
 \item There exists a unique degree two hypersurface $Q_n \subseteq \Pro^4$ containing
$X_n$.
\item Let $2 \le t \le q_0$ be an integer. The degree $t$ hypersurfaces of
$\Pro^4$ containing
$X_n$ are exactly those containing $Q_n$. Moreover, they form an $\F$-vector
space of dimension
$\binom{t+4}{4}-\binom{t+2}{4}$.
\item The previous result is false for $t>q_0$. Indeed, there exist at least
four linearly 
independent degree $q_0+1$
hypersurfaces of $\Pro^4$ containing $X_n$, and not containing $Q_n$.
\item Let $t \ge 2q_0+1$ be an integer. The degree $t$ hypersurfaces of $\Pro^4$ containing $X_n$
form an $\F$-vector space of dimension $\binom{t+4}{4}-t(q+2q_0+1)-1+g_n$.
\end{enumerate}
\end{theoremmm}
The theorem provides an interesting geometric characterization of the small-degree
hypersurfaces of $\Pro^4$ containing $X_n$. Moreover it is proved that such a
characterization cannot be extended to higher-degree hypersurfaces.
\begin{remark}
 Two linearly
independent hypersurfaces of $\Pro^4$ containing $X_n$ and not containing $Q_n$ 
appear in \cite{AMS}, page 4 (see also
Example \ref{deg2}).
\end{remark}

Section \ref{seccalcoli} and Section \ref{sec2} are 
dedicated to preliminary
results. In particular, in Section \ref{seccalcoli} we derive explicit formulas
for the dimension of any Riemann-Roch space of the form $L(t(q+2q_0+1))$, $t \in
\Z_{\ge 0}$. On the other hand, in Section \ref{sec2} we consider
some multiplication maps of geometric interest, and study their properties. The
computational results are interpreted from a geometric point of view in Section
\ref{geom}, leading to the main goal of the paper.

\begin{remark}
 The linear series $|(q+2q_0+1)\pi^{-1}(P_\infty)|$ here considered is 
of deep interest in the literature. Its properties can be
used
to characterize Suzuki curves in terms
of the genus and the number of rational points (see \cite{HKT}, Theorem
10.102).
\end{remark}

\section{Geometry on the Weierstrass semigroup} \label{sec1}

Given a Suzuki curve $S_n$ and an integer $m \ge 0$, we denote by $L(mP_\infty)$
the vector space of the rational functions on $S_n$ whose pole order at
$P_\infty$
is at most $m$, i.e., the Riemann-Roch space associated to
the
divisor $mP_\infty$
on $S_n$.
We recall that the Weierstrass semigroup $H(P_\infty)$ associated to
$P_\infty$ is precisely the set of non-gaps at $P_\infty$. In other words,
$H(P_\infty)$ is the set of all the $m \in \Z_{\ge 0}$ such that there exists a
rational function in $L(mP_\infty) \setminus L((m-1) P_\infty)$.

\begin{remark}\label{banale}
 Since, for any $m \ge 0$, we have $0 \le L((m+1)P_\infty)-L(mP_\infty) \le 1$,
by 
definition of Weierstrass semigroup we clearly get $\dim_{\F} L(mP_\infty) = 
|\{ s \in H(P_\infty) : s \le m \} |$.
\end{remark}

\begin{lemma}[\cite{MA}, Lemma 3.1]
 \label{weisem}Let $H(P_\infty)$ be the Weierstrass semigroup defined above.
Then
 $H(P_\infty)= \langle q, q+q_0, q+2q_0, q+2q_0+1
\rangle$.
\end{lemma}

\begin{notation}\label{norma}
For any
$(a,b,c,d) \in \Z_{\ge 0}^4$, we set $$\lVert(a,b,c,d)\rVert:=
aq+b(q+q_0)+c(q+2q_0)+d(q+2q_0+1).$$
\end{notation}

\begin{notation}\label{norm}
The normalization  of $S_n$ will be always
denoted by $C_n$. It is well-known (see for instance \cite{fulton}, Section 7.5) that $C_n$ is a smooth abstract 
curve which
is birational to $S_n$. The normalization morphism $\pi:C_n \to S_n$ is here
injective.
Hence we will simply write $P_\infty$ instead of $\pi^{-1}(P_\infty)$.
\end{notation}

Here we focus on $C_n$ curves and study linear systems on them of the form $|mP_\infty|$, 
providing a characterization of the very ample ones.

\begin{lemma}\label{spanned}
 Let $m$ be a positive integer. The linear system
$|mP_\infty|$
is spanned by its global sections if and only if $m \in H(P_\infty)$.
\end{lemma}

\begin{proof}
This is a well-known property of the one-point
Weierstrass
semigroup $H(P_\infty)$ (notice that $C_n$ is smooth).
\end{proof}

\begin{proposition}\label{verya}
 Let $m$ be a positive integer. The linear system $|mP_\infty|$ is
very ample if and only if $m \in H(P_\infty)$ and $m-1 \in H(P_\infty)$.
\end{proposition}
\begin{proof}
If $|mP_\infty|$ is
very ample, then it is obviously spanned by its global sections. Hence, by Lemma
\ref{spanned}, we get $m \in H(P_\infty)$. Let $r:=\dim_{\F} L(mP_\infty)$ and
denote
by $\varphi_m:C_n \to \Pro^{r-1}$ the morphism induced by $mP_\infty$. The
linear system
$|mP_\infty|$ is very ample if and only if $\varphi_m$ is injective with
non-zero
differential at any point of $C_n$.

\begin{enumerate}
 \item[($\Rightarrow$)]
Assume that the linear
system $|mP_\infty|$ is very ample. In particular, $\varphi_m$ must have
non-zero
differential at $P_\infty$. This implies the existence of a rational function $f
\in L(mP_\infty)$
whose vanishing order at $P_\infty$ is exactly one. Since $m \in H(P_\infty)$,
this implies
$m-1 \in H(P_\infty)$. 

\item[($\Leftarrow$)]
On the other hand, assume $m,m-1 \in H(P_\infty)$. We clearly have
$m \ge q+2q_0+1$. 
As in Notation \ref{norm}, let $\pi:C_n \to S_n$
denote the normalization morphism of $S_n$.
Since  
$(x)_\infty=qP_\infty$ and
$(y)_\infty=(q+q_0)P_\infty$ (see \cite{HH}, Proposition 1.3), we have $\{
1,x,y\} 
\subseteq L(mP_\infty)$. 
 Hence the linear system 
$|mP_\infty|$ contains
the linear system spanned by $\{ 1,x,y\}$, which induces the 
composition
of $\pi$ with the inclusion $S_n \hookrightarrow \Pro^2$.
Since $P_\infty$ is the only singular point of
$S_n$, the morphism $\varphi_m$ is injective with non-zero differential at any
point of
$C_n \setminus \{P_\infty\}$. Therefore, in order to prove that 
$|mP_\infty|$ is very ample, it is necessary and 
sufficient 
to show that
$\dim_{\F} L((m-2)P_\infty)= \dim_{\F} L(mP_\infty)-2$.
Since $m,m-1 \in H(P_\infty)$, this condition is clearly satisfied.
\end{enumerate}
\end{proof}

\begin{remark}
 Proposition \ref{verya} shows that the smallest projective space in which $C_n$
can
be embedded by a one-point linear system $|mP_\infty|$ is $\Pro^4$. 
\end{remark}

\section{Riemann-Roch spaces of Suzuki curves} \label{seccalcoli}
In this section we provide an explicit formula for the dimension of any 
Riemann-Roch space of the form
$L(t(q+2q_0+1)P_\infty)$, $t \in \Z_{\ge 0}$.
Since the Weierstrass semigroup $H(P_\infty)$ is known (Proposition
\ref{weisem}), the dimension of $L(mP_\infty)$ is also known, \textit{in
principle}, for any $m \ge 0$. On the other hand, deriving easy-handable
expressions from the semigroup's data is not completely trivial.
Explicit formulas and their combination are key-points in the proofs of this
paper. The main results of the Section are Proposition \ref{dimension} and
Proposition \ref{base}, whose proofs are splitted in some preliminary lemmas.

\begin{lemma}\label{lem1}
 Let $(a,b,c,d) \in \Z_{\ge 0}^4$ and let $t \le q_0-1$ be a positive integer.
The following two facts are equivalent:
\begin{enumerate}
 \item[(A)] $\lVert (a,b,c,d) \rVert \le t(q+2q_0+1)$,
 \item[(B)] $a+b+c+d \le t$.
\end{enumerate}
\end{lemma}
\begin{proof}
 Assume $a+b+c+d \le t$. Then $\lVert (a,b,c,d) \rVert \le (a+b+c+d)(q+2q_0+1)
\le t(q+2q_0+1)$. On the other hand, we may note that
$q>q-q_0-1=2(q_0-1)q_0+q_0-1 \ge 2tq_0+t$. Hence $(t+1)q>t(q+2q_0+1)$. If
$a+b+c+d >t$ then $a+b+c+d \ge t+1$. As a consequence, we get $\lVert (a,b,c,d) \rVert \ge
(a+b+c+d)q \ge (t+1)q> t(q+2q_0+1)$.
\end{proof}

\begin{lemma}\label{uniq}
 Let $(a',b',c',d') \in \Z_{\ge 0}^4$. Choose any integer $t$ with $1 \le t \le
q_0-1$  and assume
 $\lVert (a',b',c',d')\rVert \le t(q+2q_0+1)$. There
exists a unique 4-tuple $(a,b,c,d) \in \Z_{\ge 0}^4$ with $b \in \{ 0,1 \}$ and
$\lVert
(a,b,c,d)\rVert =\lVert (a',b',c',d')\rVert$.
\end{lemma}

\begin{proof}
 First of all, we prove the existence. Write $b'=2\beta+B$, with $\beta \ge 0$
and $r \in \{0,1\}$, and set $a:=a'+\beta$, $b:=B$, $c:=c'+\beta$, $d:=d'$.
Now we prove the uniqueness. Assume that there exist
$(a_1,b_1,c_1,d_1), (a_2,b_2,c_2,d_2) \in \Z_{\ge 0}^4$ such that:
\begin{itemize}
 \item[(A)] $b_1,b_2 \in \{ 0,1 \}$,
 \item[(B)] $\lVert (a_1,b_1,c_1,d_1)\rVert = \lVert (a_2,b_2,c_2,d_2)\rVert$,
 \item[(C)] $\lVert (a_1,b_1,c_1,d_1)\rVert , \lVert (a_2,b_2,c_2,d_2)\rVert
\le t(q+2q_0+1)$.
\end{itemize}
 As in the proof of Lemma \ref{lem1}, we have $(t+1)q>t(q+2q_0+1)$. 
Condition (iii) implies, in particular, $c_1,c_2,d_1,d_2 \le
 t \le q_0-1$. Condition (ii) is equivalent to 
\begin{equation}\label{original}
  (a_1-a_2)q+(b_1-b_2)(q+q_0)+(c_1-c_2)(q+2q_0)+(d_1-d_2)(q+2q_0+1)=0. 
\end{equation}
Reducing modulo $q_0$, we have $d_1-d_2 \equiv 0 \mod q_0$. Since $-q_0+1
 \le d_1,d_2 \le q_0-1$ we deduce $d_1=d_2$. Hence equation (\ref{original})
becomes
\begin{equation}\label{original2}
 (a_1-a_2)q+(b_1-b_2)(q+q_0)+(c_1-c_2)(q+2q_0)=0.
\end{equation}
Reducing modulo $2q_0$, we obtain $(b_1-b_2)q_0 \equiv 0 \mod 2q_0$. Since
$b_1,b_2 \in \{ 0,1 \}$, one gets $b_1=b_2$.
By substitution into equation (\ref{original2}), we may write
\begin{equation}\label{original3}
 (a_1-a_2)q+(c_1-c_2)(q+2q_0)=0.
\end{equation}
Reducing modulo $q$, we get $(c_1-c_2)2q_0 \equiv 0 \mod q$. Since $q=2q_0^2$
and $c_1,c_2 \le q_0-1$, we conclude $c_1=c_2$. Clearly $a_1=a_2$ at this
point. 
\end{proof}

\begin{remark}
 The uniqueness argument in the proof of the previous lemma is applied 
also in \cite{AMS},
Proposition
3.7, to get an analogous result.
\end{remark}

The following lemma summarizes some trivial enumeration facts we are going to
apply. 
A proof can be easily obtained by induction.
\begin{lemma}\label{comb}
Let $h$ be a positive integer. The following formulas hold.
\begin{enumerate}
 \item[(A)] $\sum_{i=0}^h i =h(h+1)/2$.
\item[(B)]  $\sum_{i=0}^h i^2 =h^3/3+h^2/2+h/6$.
\item[(C)] Let $\mathcal{T}_h$ be the set of all the 3-tuple $(a,b,c) \in
\Z_{\ge 0}^3$ 
satisfying $a+b+c=h$. We have $|\mathcal{T}_h|=(h+1)(h+2)/2$.
\end{enumerate}
\end{lemma}

\begin{proposition} \label{dimension}
 Let $t$ be a non-negative integer and let $g_n=q_0(q-1)$ be the genus
of the Suzuki curve $S_n$ (see the Introduction). The dimension of the one-point Riemann-Roch space
$L(t(q+2q_0+1)P_\infty)$ is given by the following formulas:

$$\dim_{\F} L(t(q+2q_0+1)P_\infty)= \left\{  
 \begin{array}{ll}
  4t+1 & \mbox{ if $t=0$ or $t=1$,} \\
    \binom{t+4}{4}-\binom{t+2}{4} & \mbox{ if $2 \le t \le q_0-1$,} \\
  t(q+2q_0+1)+1-g_n+\binom{2q_0-t+2}{4}-\binom{2q_0-t}{4} & \mbox{ if $q_0 \le t
\le 2q_0-4$,} \\
  t(q+2q_0+1)+6-g_n & \mbox{ if $t=2q_0-3$,} \\
  t(q+2q_0+1)+2-g_n & \mbox{ if $t=2q_0-2$,} \\
  t(q+2q_0+1)+1-g_n & \mbox{ if $t \ge 2q_0-1$.}
 \end{array}\right.\ $$
\end{proposition}

\begin{proof}
We recall (Remark \ref{banale}) that $\dim_{\F}L(t(q+2q_0+1))$ is exactly the
cardinality
of the set $H_t(P_\infty):=\{ s \in
H(P_\infty) : s \le t(q+2q_0+1)\}$. 
The proof is divided into five steps.
\begin{enumerate}
\item[(A)] If $t=0,1$ the dimension is easily computed by hands (Lemma \ref{weisem}).
\item[(B)] Assume $2 \le t \le q_0-1$.  Combining Lemma \ref{lem1} and Lemma
\ref{uniq} we see that, for any $t \in \{2,...,q_0-1 \}$, the cardinality of
$H_t(P_\infty)$ may be computed as
\begin{equation*}|H_t(P_\infty)|= |\{ (a,b,c,d) \in \Z_{\ge 0}^4 : b\in \{ 0,1\} \mbox{ and }
a+b+c+d\le t \} |.\end{equation*}
Hence, following the notation of Lemma \ref{comb}, we write
\begin{eqnarray*}
 |H_t(P_\infty)| &=& \sum_{h=0}^t \ |\{ (a,b,c,d) \in \Z_{\ge 0}^4 : b\in \{
0,1\} \mbox{ and }
a+b+c+d= h \} | \\
&=&\sum_{h=0}^t |\mathcal{T}_h| + \sum_{h=1}^t |\mathcal{T}_{h-1}| \\
&=& |\mathcal{T}_t|+2\sum_{h=0}^{t-1} |\mathcal{T}_h| \\
&=& (t+1)(t+2)/2 + \sum_{h=0}^{t-1} h^2+3h+2 \\
&=& (2t^3+9t^2+13t+6)/6 \\
&=& \binom{t+4}{4}-\binom{t+2}{4},
\end{eqnarray*}
which is the expected formula.
\item[(C)] Since  the
genus of $S_n$ is $g_n=q_0(q-1)$, we compute 
 $2g_n-2=2(q_0-1)(q+2q_0+1)$. Hence, for $t\ge 2q_0-2$, 
the dimension of $L(t(q+2q_0+1))$ is
given by a trivial application of the Riemann-Roch Theorem and the fact that
$\dim_{\F} L(0)= 1$.
\item[(D)] Here we assume $q_0\le t \le 2q_0-4$ and set
$D_t:=t(q+2q_0+1)P_\infty$. A canonical divisor
on $S_n$ is $K=(2g_n-2)P_\infty \sim 2(q_0-1)(q+2q_0+1)P_\infty$. See also
\cite{AMS} for details.
We have a linear equivalence of divisors
$$K-D_t \sim (2q_0-2-t)(q+2q_0+1)P_\infty.$$
Since $2 \le 2q_0-2-t\le q_0-1$, thanks to step (B) we are
able to explicitly compute $\dim_{\F} L(K-D_t)$ and obtain $\dim_{\F} L(D_t)$
by
applying the Riemann Roch Theorem as follows:
$$\dim_{\F} L(D_t)=t(q+2q_0+1)+1-g_n+\binom{2q_0-t+2}{4}-\binom{2q_0-t}{4}.$$
\item[(E)] Finally, assume $t=2q_0-3$ and set $D:=(2q_0-3)(q+2q_0+1)P_\infty$. We have a linear
equivalence
$K-D \sim (q+2q_0+1)P_\infty$ and so, by step (A), the dimension of $L(D)$ is
again computed 
by the Riemann-Roch Theorem. \qedhere
\end{enumerate}
\end{proof}

We conclude this section  providing an explicit
monomial basis of any Riemann-Roch space
$L(mP_\infty)$, $m \ge 0$. The following preliminary result generalizes Lemma
\ref{uniq}.

\begin{lemma}\label{uniq2}
 Let $(a',b',c',d') \in \Z_{\ge 0}^4$.  There
exists a unique $(a,b,c,d) \in \Z_{\ge 0}^4$ which satisfies the following
properties: 
$$0 \le b \le 1, \ \ \ 0 \le c \le q_0-1, \ \ \ 0 \le d \le q_0-1, \ \ \ \lVert
(a,b,c,d)\rVert =\lVert (a',b',c',d')\rVert.$$
\end{lemma}
\begin{proof}
To prove the uniqueness we may apply the same argument as Lemma \ref{uniq},
which uses
only our hypothesis on $b$, $c$ and $d$.
Let us prove the existence. Write $d'=\delta q_0+D$ with $0 \le D \le q_0-1$ and
set
$(a_1,b_1,c_1,d_1):=(a'+\delta q_0, b'+\delta, c',D)$.
Write $b_1=2\beta+B$ with $0 \le B \le 1$, and set
$(a_2,b_2,c_2,d_2):=(a_1+\beta, B, c_1+\beta,D)$.
Write $c_2=\gamma q_0+C$ with $0 \le C \le q_0-1$, and define
$$(a,b,c,d):=(a_2+\gamma q_0+\gamma, b_2, C, d_2)=(a'+\delta q_0 + \gamma q_0 +
\beta + \gamma,
B,C,D).$$
It is easily checked that $(a,b,c,d)$ has the expected properties.
\end{proof}

\begin{definition}\label{defff}
 Following \cite{HH} and \cite{MA}, we define the rational functions
$v:=y^{2q_0}+x^{2q_0+1}$ and $w:= y^{2q_0}x+v^{2q_0}$. The pole divisors
of $x,y,v,w$ are computed in \cite{HH}, Proposition 1.3:
$$(x)_\infty=qP_\infty, \ \ \ (y)_\infty=(q+q_0)P_\infty, \ \ \ 
(v)_\infty= (q+2q_0)P_\infty, \ \ \ (w)_\infty=(q+2q_0+1)P_\infty.$$
\end{definition}

\begin{remark} \label{ordinenorma}
 From the pole divisors given in the previous definition we seet that,
for any $(a,b,c,d) \in \Z_{\ge 0}^4$, the pole order of $x^ay^bv^cw^d$ at
$P_\infty$ is exactly $\lVert (a,b,c,d) \rVert$. 
\end{remark}

\begin{proposition} \label{base}
 Let $m \ge 0$ be an integer. A basis of the Riemann-Roch space
$L(mP_\infty)$ is given by all the rational functions $x^ay^bv^cw^d$
such that:
$$a,b,c,d \in \Z_{\ge 0}, \ \ \ 0 \le b \le 1, \ \ \ 0 \le c,d \le q_0-1, \ \ \
\lVert (a,b,c,d) \rVert \le m.$$
\end{proposition}
\begin{proof}
By Lemma \ref{uniq2}, such rational functions have different pole orders at
$P_\infty$.
In particular, they are linearly independent. By Remark \ref{ordinenorma}, they
all
belong to $L(mP_\infty)$. Finally, by definition of $H(P_\infty)$ and Lemma \ref{uniq2}, 
their number
is
$\dim_{\F} L(mP_\infty)$. Hence they form a basis of the Riemann-Roch space
$L(mP_\infty)$.
\end{proof}

\section{Multiplication maps and their geometry}
\label{sec2}

Let $S_n$ be the Suzuki curve defined in the Introduction and let $\pi: C_n \to
S_n$ be its
normalization
(see Notation \ref{norm}). 
By Proposition
\ref{verya}, the linear system $|(q+2q_0+1)P_\infty|$ defines an embedding
$\varphi_{q+2q_0+1}: C_n \to \Pro^4$. We set $X_n:=\varphi_{q+2q_0+1}(C_n)$, 
a smooth curve 
of degree $q+2q_0+1$ in $\Pro^4$.

\begin{definition} \label{multipl}Given non-negative integers $a$, $b$ and $t$,
we will denote by
$\mu(a,b)$ and $\mu_t(a)$, respectively, the multiplication maps
$$\mu(a,b): L(aP_\infty) \otimes L(bP_\infty) \to L((a+b)P_\infty), \ \ \ \ \
\mu_t(a): {L(aP_\infty)}^{\otimes t}
\to L(taP_\infty).$$  Since  in the function field
defined by $S_n$ multiplication is commutative, each of
the maps $\mu_t(a)$ induces a multiplication map $\sigma_t(a):S^t(L(aP_\infty))
\to L(taP_\infty)$,
where $S^t(L(aP_\infty))$ denotes the $t$-th power of the symmetric tensor
product. 
\end{definition}

\begin{remark}
 This section is rather technical. Here we study the surjectivity
of the multiplication maps $\sigma_t(q+2q_0+1)$, $t \ge 1$, introduced in
Definition \ref{multipl}.
Interesting geometric applications will be shown later in the
paper. The
main results of
this section are Proposition \ref{chiavesurg} and its consequences (Corollary \ref{projn}). The proof
of the cited proposition is splitted in Lemmas \ref{surgia}, \ref{surgalta},
\ref{indep} and \ref{steppp}.
\end{remark}

\begin{lemma} \label{surgia}
 Let $\alpha$ and $\beta$ be non negative integers such that $\alpha+\beta \le
q_0-1$.
The multiplication map $\mu(\alpha(q+2q_0+1), \beta (q+2q_0+1))$ of Definition
 \ref{multipl} is surjective.
\end{lemma}
\begin{proof}
Since $\alpha$ and $\beta$ play interchangeable roles and the case $\alpha=0$ is
trivial,
we may assume $\beta \ge \alpha >0$. Keep on mind Proposition \ref{base} and
consider
a basis element, $x^ay^bv^cw^d$, of the Riemann-Roch space
$L((\alpha+\beta)(q+2q_0+1)P_\infty)$. We clearly have 
\begin{equation} \label{generaleq}
aq+b(q+q_0)+c(q+2q_0)+d(q+2q_0+1) \le (\alpha+\beta)(q+2q_0+1).
\end{equation}
Since $\alpha+\beta \le q_0-1$, we get
$\alpha+\beta \le q_0-1 < 2q_0^2/(2q_0+1)=q/(2q_0+1)$.
As a consequence, $(\alpha+\beta)(2q_0+1)<q$, i.e.,
$q(\alpha+\beta+1)>(\alpha+\beta)(q+2q_0+1)$.
By inequality (\ref{generaleq}) we have, in particular,
$(a+b+c+d)q \le (\alpha+\beta)(q+2q_0+1)<q(\alpha+\beta+1)$. Dividing by $q$ one
obtains
$a+b+c+d < \alpha+\beta +1$ and so $a+b+c+d \le \alpha+\beta$. Now we write
$(a,b,c,d)=(a_1,b_1,c_1,d_1)+(a_2,b_2,c_2,d_2)$ with
$a_1+b_1+c_1+d_1 \le \alpha$  and $a_2+b_2+c_2+d_2 \le \beta$.
It follows $\lVert (a_1,b_1,c_1,d_1) \rVert \le \alpha(q+2q_0+1)$, 
$\lVert (a_2,b_2,c_2,d_2) \rVert \le \beta(q+2q_0+1)$ and so
$$x^ay^bv^cw^d=\mu(\alpha(q+2q_0+1),\beta(q+2q_0+1)) \ \
(x^{a_1}y^{b_1}v^{c_1}w^{d_1} \otimes 
x^{a_2}y^{b_2}v^{c_2}w^{d_2}).$$
In other words, a generic basis element $x^ay^bv^cw^d \in
L((\alpha+\beta)(q+2q_0+1)P_\infty))$
 is in the image of $\mu(\alpha(q+2q_0+1),\beta(q+2q_0+1))$, as claimed.
\end{proof}

\begin{lemma}\label{surgalta}
  Let $t \ge 1$ be an integer. We follow the notation of Lemma \ref{comb}. Let
$(a,b,c,d) \in \Z_{\ge 0}^4$ with $a+b+c+d \le t$. There exist 4-tuple
${\{ (a_i,b_i,c_i,d_i) \}}_{i=1}^t \subseteq \mathcal{T}_1$ such that
$$(a,b,c,d)=\sum_{i=1}^t (a_i,b_i,c_i,d_i).$$
 \end{lemma}
\begin{proof}
 We use induction on $t$. If $t=1$ then we take
$(a_1,b_1,c_1,d_1):=(a,b,c,d)$. Now assume $a+b+c+d \le t+1$. If
$a+b+c+d \le t$ then, by inductive hypothesis, we write
$(a,b,c,d)=\sum_{i=1}^t (a_i,b_i,c_i,d_i)$ with 
${\{ (a_i,b_i,c_i,d_i) \}}_{i=1}^t \subseteq \mathcal{T}_1$. As a consequence, we may 
define the 4-tuple $(a_{t+1},
b_{t+1}, c_{t+1}, d_{t+1}):=(0,0,0,0) \in \mathcal{T}_1$ and obtain
$(a,b,c,d)=\sum_{i=1}^{t+1} (a_i,b_i,c_i,d_i)$ . On the other hand, if $a+b+c+d=t+1$ then 
one among
$a,b,c,d$ must be positive. Assume without restriction $a>0$. Then, by
induction, $(a-1,b,c,d)= \sum_{i=1}^t (a_i,b_i,c_i,d_i)$ with 
${\{ (a_i,b_i,c_i,d_i) \}}_{i=1}^t \subseteq \mathcal{T}_1$. By setting $(a_{t+1},
b_{t+1}, c_{t+1}, d_{t+1}):=(1,0,0,0) \in \mathcal{T}_1$, we have
$(a,b,c,d)=\sum_{i=1}^{t+1} (a_i,b_i,c_i,d_i)$, and the lemma is proved.
\end{proof}

\begin{lemma}\label{indep}
 Let $m$ be a positive integer. Let $\{ f_1,...,f_h\} \subseteq L(mP_\infty)$ be
a set of rational functions such that for any $s \in H(P_\infty)$, with $s \le
m$, there exists a $1 \le j \le h$ such that $(f_j)_{\infty}=s$. Then $\{
f_1,...,f_h\}$ is a generating set of $L(mP_\infty)$.
\end{lemma}

\begin{proof}
 For any $s \in H(P)$, with $s \le m$, there exists a $1 \le j_s \le h$ such
that 
 $(f_{j_s})_{\infty}=sP_\infty$. Hence $\mathcal{B}_m:=\{ f_{j_s} : s \in
H(P_\infty), s \le m\}$ is a set of linearly
independent elements of $L(mP_\infty)$ whose cardinality is  $\dim_{\F}
L(mP_\infty)$. Indeed, its elements are rational
functions whose evaluations at $P_\infty$ are distinct. Hence $\mathcal{B}_m$ 
is a basis of 
$L(mP_\infty)$. We conclude by observing that $\{ f_1,...,f_h\}$ contains
$\mathcal{B}_m$.
\end{proof}

\begin{lemma} \label{steppp}
 Let $t \ge 2q_0+1$ be an integer. For any $s\in H(P_\infty)$, with $s \le
t(q+2q_0+1)$,  there exists a 4-tuple $(a,b,c,d)\in \Z_{\ge 0}^4$
such
that $\lVert (a,b,c,d)\rVert = s$ and $a+b+c+d\le t$. 
\end{lemma}

\begin{proof}
The argument is divided into two steps.
 \begin{enumerate}
 \item[(A)] Here we assume $s\le tq$ and take any 4-tuple $(a,b,c,d)$ such that
$\lVert (a,b,c,d)\rVert = s$; such a 4-tuple exists because $s \in
H(P_\infty)$.
If $a+b+c+d \ge t+1$, then we clearly have the contradiction $s=\lVert
(a,b,c,d)\rVert \ge (t+1)q >s$. Hence $a+b+c+d \le t$, and we
are done.

\item[(B)] Now assume $s>tq$. Write $s= \alpha(q+2q_0+1)-\beta$ with $\alpha,
\beta \in \Z_{\ge 0}$ and $0\le \beta \le
q+2q_0$. Since $s \le t(q+2q_0+1)$, we  have $0 \le \alpha \le t$.
Set: \begin{equation*}
      e_1:= \Big\lfloor \frac{\beta}{2q_0+1}\Big\rfloor, \ \ \ \
\ e_2:=\Big\lfloor
\frac{\beta -e_1(2q_0+1)}{q_0+1}\Big\rfloor, \ \ \ \ \  e_3:=
\beta -e_1(2q_0+1)-e_2(q_0+1).
     \end{equation*}
Notice
that $e_2\in \{0,1\}$ and $e_3\le q_0$.
Since $\beta \le q+2q_0 =(2q_0+1)q_0$, we get $e_1\le q_0$ and the equality
holds if and only if $\beta=q+2q_0$. In this case we have $e_2=\lfloor
q_0/(q_0+1)\rfloor=0$. Hence, in any case,  $e_1+e_2+e_3 \le 2q_0$.
We may also observe
$$s >tq \ge (2q_0+1)q > (2q_0-1)(q+2q_0+1).$$
Since $s=\alpha(q+2q_0+1)-\beta$ with $0 \le \beta \le q+2q_0$, we deduce $\alpha \ge 2q_0 \ge 
e_1+e_2+e_3$. Hence $e_1+e_2+e_3 \le \alpha \le t$ and we may take
$a:= e_1$, $b:=e_2$, $c:=
e_3$ and $d:= \alpha-e_1-e_2-e_3$ to conclude the proof. \qedhere
\end{enumerate}
\end{proof}

\begin{proposition} \label{chiavesurg}
 Let $t$ be a positive integer and let $\sigma_t(q+2q_0+1)$ be as in
 Definition \ref{multipl}.
\begin{enumerate}
 \item If $1 \le t \le q_0$ then $\sigma_t(q+2q_0+1)$ is surjective.
\item If $t \ge 2q_0+1$ then  $\sigma_t(q+2q_0+1)$ is surjective.
\end{enumerate}
\end{proposition}
\begin{proof}
Let us divide the proof into three steps.
\begin{enumerate}
 \item[(A)] Here we assume $1 \le t \le q_0-1$. In the notations of Definition \ref{multipl}, 
 the image of the map
$\sigma_t(q+2q_0+1)$ and the image of the map $\mu_t(q+2q_0+1)$ coincide.
Moreover, the image of $\mu_t(q+2q_0+1)$ contains the image of 
$\mu(q+2q_0+1, (t-1)(q+2q_0+1))$.
Since $t =1+(t-1) \le q_0-1$, by  Lemma \ref{surgia} the map
$\mu(q+2q_0+1, (t-1)(q+2q_0+1))$ 
is surjective, and we are done.

\item[(B)] Assume $t \ge 2q_0+1$. We recall that $L(q+2q_0+1)$ has $\{1, x,y,v,w\}$ as a basis. 
Moreover,
$1, x,y,v,w$ have the following pole divisors:
$${(1)}_\infty=0, \ \ \ \ \ {(x)}_\infty=q, \ \ \ \ \ {(y)}_\infty=q+q_0, \
\ \ \ \ {(v)}_\infty=q+2q_0, \ \ \ \ \ {(w)}_\infty=q+2q_0+1.$$
Take any $s \in H(P_\infty)$ with $s \le
t(q+2q_0+1)$. By Lemma \ref{steppp} there exists a 4-tuple $(a,b,c,d)\in \Z_{\ge 0}^4$
such
that $\lVert (a,b,c,d)\rVert = s$ and $a+b+c+d\le t$. 
Thanks to Lemma \ref{surgalta}, we write $(a,b,c,d)=\sum_{i=1}^t (a_i,b_i,c_i,d_i)$,
with $(a_i,b_i,c_i,d_i) \in \mathcal{T}_1$ for any $i \in \{ 1,...,t\}$. Hence, for any $i\in \{ 1,...,t\}$, we have
$x^{a_i}y^{b_i}v^{c_i}w^{d_i} \in L((q+2q_0+1)P_\infty)$. Moreover, 
$$\sigma_t (q+2q_0+1) \left( \bigotimes_{i=1}^t \ x^{a_i}y^{b_i}v^{c_i}w^{d_i}\right)=x^ay^bv^cw^d$$
is a rational function in the image of $\sigma_t$ whose pole divisor is exactly $sP_\infty$.
Notice that $s$ is arbitrary in $H(P_\infty)$ with $s \le t(q+2q_0+1)$. Hence,
by Lemma \ref{indep}, the image of $\sigma_t(q+2q_0+1)$ spans the vector space 
$L(t(q+2q_0+1)P_\infty)$, i.e., $\sigma_t(q+2q_0+1)$
is surjective. 

\item[(C)] Let us consider the case $t=q_0$. 
By Lemma \ref{indep}, it is enough to prove that for any
$s\in H(P_\infty)$, with $s \le
q_0(q+2q_0+1)$,  there exists a rational function, say $f$, in the image of
$\sigma_{q_0}(q+2q_0+1)$ with the property  $(f)_\infty=sP_\infty$.
Write $s=\lVert (a,b,c,d) \rVert$ for a certain $(a,b,c,d) \in \Z_{\ge 0}^4$.
\begin{enumerate}
 \item[(C.1)] If $s \le (q_0-1)(q+2q_0+1)$ then, by Lemma \ref{lem1}, we 
conclude $a+b+c+d \le q_0-1<q_0$. In this case we are done, as in step (B).
\item[(C.2)] Assume $s>(q_0-1)(q+2q_0+1)$. If $a+b+c+d \le q_0$ then, as above, the result 
is proved. 
Hence we will
assume $a+b+c+d \ge q_0+1$ for the rest of the proof. If $a+b+c+d \ge q_0+2$ then
$\lVert (a,b,c,d) \rVert \ge (q_0+2)q>q_0(q+2q_0+1)$. As a consequence, 
we have $a+b+c+d \le q_0+1$
and so $a+b+c+d=q_0+1$.
Assume $a\le q_0-1$. Then $b+c+d \ge 2$ and so $\lVert (a,b,c,d) \rVert
\ge (q_0+1)q+2q_0>q_0(q+2q_0+1)$, a contradiction. It follows $a \in \{ q_0,q_0+1\}$, and
we can study the two cases separately.
\begin{itemize}
\item If $a=q_0+1$ then clearly $b=c=d=0$ and so $x^{q_0+1}$ is a rational function with
the expected pole divisor. Working modulo the equation of $S_n$, we have
$x^{q_0+1}=v^{q_0}-y$ (see \cite{AMS}, page 4). Since $v^{q_0}$ and $y$
trivially
belong to the image of $\sigma_{q_0}(q+2q_0+1)$, $x^{q_0+1}$ also belongs to such image.
\item If $a=q_0$ and $(c,d) \neq (0,0)$, then $\lVert (a,b,c,d) \rVert \ge 
q_0q+(q+2q_0)>q_0(q+2q_0+1)$, a contradiction. It follows $c=d=0$ and $b=1$. Notice that
 $x^{q_0}y$ 
is a rational function with
the expected pole divisor. Moreover, $x^qy=w^{q_0}-v$ (again \cite{AMS}, page
4) and so we conclude as in the previous
 step. \qedhere
\end{itemize}
\end{enumerate}
\end{enumerate}
\end{proof}

\begin{corollary}\label{projn}
 Let $t$ be an integer. If $1 \le t \le q_0$ or $t \ge 2q_0+1$ then
the restriction 
map
of cohomology groups 
$\rho_t: H^0(\Pro^4, \Ol_{\Pro^4}(t)) \to H^0(X_n, \Ol_{X_n}(t))$
is surjective.
\end{corollary}
\begin{proof}
 Since the embedding of curves $\varphi_{q+2q_0+1}:C_n \to X_n$ is induced by
the linear system
$|(q+2q_0+1)P_\infty|$, the pull-back
bundle of $\Ol_{X_n}(1)$  through $\varphi_{q+2q_0+1}$ is 
that associated to the linear system $|(q+2q_0+1)P_\infty|$.
By Proposition \ref{chiavesurg}, the restriction map 
$S^t(H^0(\Pro^4, \Ol_{\Pro^4}(1))) \to H^0(X_n, \Ol_{X_n}(t))$ is surjective.
The thesis follows. 
\end{proof}

\section{The smooth model of a Suzuki curve in $\Pro^4$} \label{geom}

Here we study the
geometric properties of the smooth model $X_n \subseteq \Pro^4$ of a Suzuki
curve $S_n$. We apply the computational results derived
in the previous
parts of the paper in order to count the
hypersurfaces of $\Pro^4$ containing $X_n$ (Theorem \ref{ttt}). Moreover, we
provide an explicit geometric characterization of
those of small degree (Corollary \ref{bombetta}).

\begin{theorem} \label{ttt}
 Let $t$ be a positive integer and let $\mathcal{K}(t,X_n)$
denote the $\F$-vector space  of
all the degree $t$ 
hypersurfaces of $\Pro^4$ containing $X_n$. Let $\kappa(t,X_n)$ be 
the dimension of $\mathcal{K}(t,X_n)$. The following formulas hold:
$$ 
\kappa(t,X_n)= \left\{ \begin{array}{ll} 
                \binom{t+4}{4}-\binom{t+2}{4} & \mbox{ if $2 \le t \le q_0$,} \\
 \binom{t+4}{4}-t(q+2q_0+1)-1+g_n & \mbox{ if $t \ge 2q_0+1$.} 
               \end{array} \right.\  $$
\end{theorem}

\begin{proof}
The vector space $\mathcal{K}(t,X_n)$, whose dimension is in question, is exactly 
the kernel of the restriction
map
$\rho_t:H^0(\Pro^4, \Ol_{\Pro^4}(t)) \to H^0(X_n, \Ol_{X_n}(t))$. If  
$2 \le t \le q_0$ or $t \ge 2q_0+1$, by Corollary \ref{projn}, 
 $\rho_t$ 
is surjective. It follows
$$\kappa(t,X_n) = h^0(\Pro^4, \Ol_{\Pro^4}(t))-h^0(X_n, \Ol_{X_n}(t)) =
 h^0(\Pro^4, \Ol_{\Pro^4}(t))-\dim_{\F} L(t(q+2q_0+1)P_\infty). $$
Now use the formulas given in Proposition \ref{dimension}.
\end{proof}

Theorem \ref{ttt} allows us to geometrically charaterize 
all the small-degree hypersurfaces of $\Pro^4$ containing $X_n$. 

\begin{corollary}\label{bombetta}
 Let $X_n$ be the smooth projective model of the Suzuki curve $S_n$ in $\Pro^4$,
obtained through 
the linear system
$|(q+2q_0+1)P_\infty|$. The following facts hold.
\begin{enumerate}
 \item There exists a unique degree two hypersurface $Q_n \subseteq \Pro^4$
containing $X_n$.
\item Let $2 \le t \le q_0$ be an integer. The degree $t$ hypersurfaces of $\Pro^4$ containing
$X_n$ are exactly those containing $Q_n$. Moreover, they form an $\F$-vector
space of dimension
$\binom{t+4}{4}-\binom{t+2}{4}$.
\item There exist at least four linearly independent degree $q_0+1$
hypersurfaces of $\Pro^4$ containing $X_n$ and not containing $Q_n$.
\end{enumerate}
\end{corollary}

\begin{proof}
Let $t\ge 2$ be any integer. Since $h^1(\mathbb {P}^4,\mathcal {O}_{\mathbb {P}^4}(t-2))=0$, 
the well-known exact
sequence  of sheaves
$0\to \mathcal {O}_{\mathbb {P}^4}(t-2) \to \mathcal {O}_{\mathbb {P}^4}(t) \to
\mathcal 
{O}_{Q_n}(t)\to 0$
gives that the induced restriction map of cohomology groups
 $\rho_t': H^0(\mathbb {P}^4,\mathcal {O}_{\mathbb
{P}^4}(t))\to
 H^0(Q_n,\mathcal {O}_{Q_n}(t))$ is surjective,
and allows us to compute $h^0(Q_n,\mathcal {O}_{Q_n}(t)) =\binom{t+4}{4}
-\binom{t+2}{4}$. 
The restriction map 
$\rho_t : H^0(\mathbb {P}^4,\mathcal {O}_{\mathbb {P}^4}(t))\to 
H^0(X_n,\Ol_{X_n}(t))$ factors through $\rho_t'$. More precisely, we have a commutative diagram
of restriction maps between cohomology groups:
$$\xymatrix{
H^0(\Pro^4,\Ol_{\Pro^4}(t)) \ar[rr]^{\rho_{t}'} \ar[ddrr]_{\rho_{t}} & & 
H^0(Q_n, \Ol_{Q_n}(t)) \ar[dd]^{\rho_{t}''} \\ & & \\
 & & H^0(X_n, \Ol_{X_n}(t))}$$
(we recall that $Q_n$ contains $X_n$).
Since $\rho_t'$ is surjective,  we clearly have $\mbox{Im}(\rho_t) = \mbox{Im}(\rho_t'')$. 
Let us divide the rest of the proof into two steps.
\begin{enumerate}
 \item[(A)] Assume $2 \le t \le q_0$. We proved in Corollary \ref{projn} that the restriction 
map $\rho_t$ is
surjective. On the
other hand, as in the proof of Theorem \ref{ttt},  
$h^0(X_n, \Ol_{X_n}(t)) = \binom{t+4}{4} -\binom{t+2}{4}$. Hence $\rho_t''$ is
bijective. 
It follows
$\mbox{ker}(\rho_t) =\mbox{ker}(\rho_t')$, i.e., every degree $t$ hypersurface
of $\Pro^4$ 
contains $X_n$ if and
only if  it is a union of $Q_n$ and a degree $t-2$ hypersurface of $\Pro^4$.
\item[(B)] Assume $t=q_0+1$. Proposition \ref{dimension} and straightforward computations allow 
us to write the dimension of
the vector space $L(t(q+2q_0+1)P_\infty)$ as
$$\dim_{\F} L(t(q+2q_0+1)P_\infty)=\binom{q_0+5}{4}-\binom{q_0+3}{4}-4=
\binom{t+4}{4}-\binom{t+2}{4}-4.$$
Since $X_n$ is obtained by embedding $C_n$ through the linear system
$|(q+2q_0+1)P_\infty|$, we have also
$h^0(X_n,\Ol_{X_n}(t))=\binom{t+4}{4}-\binom{4+2}{4}-4$.
Since $\rho_{t}'$ is surjective, $\dim_{\F} 
\ker(\rho_{t}')=\binom{t+2}{4}$. As a consequence, we deduce the 
following inequality:
\begin{eqnarray*}
 \dim_{\F} \ker (\rho_{t}) &\ge& h^0(\Pro^4,\Ol_{\Pro^4}(t-2))- h^0(X_n,
\Ol_{X_n}(t))\\
&=& \binom{t+4}{4}- \left[ \binom{t+4}{4}-\binom{t+2}{4}-4\right] \\
&=& \binom{t+2}{4}+4 \\
&=& \dim_{\F} \ker (\rho_{t}')+4.
\end{eqnarray*}
Since $\ker (\rho_{t}') \subseteq \ker (\rho_{t})$, there must
exist at least four linearly independent hypersurfaces of $\Pro^4$ vanishing on $X_n$ and not
vanishing on $Q_n$, as claimed. \qedhere
\end{enumerate}
\end{proof}

\begin{example}\label{deg2}
 By Proposition \ref{base}, a basis of the Riemann-Roch space
$L((q+2q_0+1)P_\infty)$ is given by
$\{ 1,x,y,v,w\}$. Taking homogeneous coordinates $(x_1:x_2:x_3:x_4:x_5)$ in
$\Pro^4$,
we assume without loss of generality that $X_n$ is the embedding of $C_n$ 
defined by the following relations:
$$x_1/x_5=x, \ \ \ x_2/x_5=y, \ \ \ x_3/x_5=v, \ \ \ x_4/x_5=w.$$
It is easily checked that the degree two hypersurface $Q_n \subseteq \Pro^4$
defined by the
affine equation $x_2^2=x_1x_3+x_4$ contains $X_n$. By Corollary \ref{bombetta}, $Q_n$
 is the unique 
degree two
hypersurface of $\Pro^4$ containing $X_n$ (its equation is defined up to a 
scalar multiplication). The equations of two linearly independent degree $q_0+1$
hypersurfaces of $\Pro^4$ and not containing $Q_n$ appeared in the proof of 
Proposition \ref{chiavesurg}, step (C.2):
$$x^{q_0+1}=v^{q_0}-y, \ \ \ \ \  x^qy=w^{q_0}-v.$$
As pointed out in the Introduction, we find the same equations in \cite{AMS},
page 4.
\end{example}

\begin{remark}
 Lemma \ref{verya} provides an explicit characterization of all the very ample
linear systems of the form $|mP_\infty|$. We studied in details the case
$m=q+2q_0+1$, which provides the \textquoteleft smallest\textquoteright
\ possible embedding of $C_n$.   Other very ample linear systems can be
considered,
obtaining projective models of Suzuki curves in higher-dimensional projective
spaces. We notice that the smallest $m>q+2q_0+1$ such that $|mP_\infty|$ is
very ample is $2q+2q_0+1$. Moreover, $|(2q+2q_0+1)P_\infty|$ embeds $C_n$ into
$\Pro^9$. A systematic study of
higher-degree embeddings seems to be difficult.
\end{remark}

\section*{Conclusions}
In this paper we construct projetive smooth models of a plane Suzuki curve
$S_n$ through linear
systems of the form $|mP_\infty|$, where $P_\infty$ is the only singular point
of any $S_n$. 
Computational results on the Weirstrass semigroup at 
$P_\infty$ are applied in order to study in depth the smallest possible
embedding $X_n$, in
$\Pro^4$, from a geometric point of view. In particular, the small-degree
hypersurfaces of 
$\Pro^4$ containing $X_n$ are characterized in Corollary \ref{bombetta},
proving also that the same result cannot be extended to higher-degree
hypersurfaces. Moreover, high-degree hypersurfaces of $\Pro^4$ containing $X_n$
are explicitly 
counted. In order to derive 
such geometric results, here we solve some one-point Riemann-Roch
problems in the range which is
not trivially covered by the homonymous theorem, providing closed formulas.

\end{document}